\documentclass{amsart}

\usepackage{amsfonts, amssymb, amsmath, eucal, verbatim, amsthm, amscd, enumerate}
\usepackage{setspace}

\newtheorem{theorem}{Theorem}[section]

\newtheorem{lemma}[theorem]{Lemma}
\newtheorem{proposition}[theorem]{Proposition}

\theoremstyle{definition}

\theoremstyle{remark}
\newtheorem*{remark}{Remark}

\newtheorem*{acknowledgements}{Acknowledgements}
\newtheorem*{organisation}{Organisation}

\usepackage{tikz}
\usetikzlibrary{
calc,trees,shadows,positioning,arrows,chains,shapes.geometric,
decorations.pathreplacing,decorations.pathmorphing,shapes,
matrix,shapes.symbols,patterns,intersections,fit
}
\usepackage{hyperref}
\newcommand{\agl}[1]{\langle#1\rangle}
\newcommand{\C}{\mathcal{C}}

\usepackage[toc,page]{appendix}
\numberwithin{equation}{section}
\def\R{{\mathbb R}}

\newcommand{\diffp}{D_+^{\beta_{+}}}
\newcommand{\diffn}{D_-^{\beta_{-}}}
\newcommand{\lqr}{{L^q_tL^r_x}}

\newcommand{\supp}{\text{ supp }}

\author{Neal Bez}
\address[Neal Bez]{Department of Mathematics, Graduate School of Science and Engineering,
Saitama University, Saitama 338-8570, Japan}
\email{nealbez@mail.saitama-u.ac.jp}

\author{Jayson Cunanan}
\address[Jayson Cunanan]{Department of Mathematics, Graduate School of Science and Engineering,
Saitama University, Saitama 338-8570, Japan}
\email{jcunanan@mail.saitama-u.ac.jp}

\author{Sanghyuk Lee}
\address[Sanghyuk Lee]{Department of Mathematical Sciences, Seoul National University, Seoul 151-747, Korea}
\email{shklee@snu.ac.kr}


\begin{document}

\begin{abstract}
We prove smoothing estimates for velocity averages of the kinetic transport equation in hyperbolic Sobolev spaces at the critical regularity, leading to a complete characterisation of the allowable regularity exponents. Such estimates will be deduced from some mixed-norm estimates for the cone multiplier operator at a certain critical index. Our argument is not particular to the geometry of the cone and we illustrate this by establishing analogous estimates for the paraboloid.
\end{abstract}

\date{\today}

\title[Smoothing estimates at the critical regularity]{Smoothing estimates for the kinetic transport equation at the critical regularity}
%{Estimates for the kinetic transport equation in hyperbolic Sobolev spaces}
\maketitle %\thispagestyle{empty}

\section{Introduction}

For solutions $F(x,v,t)$ of the kinetic transport equation $(\partial _t +v\cdot\nabla)F = G$, estimates which capture of the positive smoothing effect of the velocity average $\int_{\mathbb{R}^d} F(x,v,t) \, \mathrm{d}\mu(v)$ traces back at least to the work of Golse--Perthame--Sentis \cite{GPS} and Golse--Lions--Perthame--Sentis \cite{GLPS}. Here, $\mu$ is a suitable measure on the space of velocities; the unit sphere $\mathbb{S}^{d-1}$ and the unit ball $\mathbb{B}^{d-1}$ are of special interest for physical reasons and so we focus our attention on these cases in this paper. The literature on such \emph{averaging lemmas} has grown significantly with regular important developments, including \cite{BBGL}, \cite{Bezard}, \cite{Bouchut2002}, \cite{BDes}, \cite{BGRevista}, \cite{BP}, \cite{DeVP}, \cite{DLM}, \cite{Gerard}, \cite{GPR}, \cite{JabinVegaParis}, \cite{JabinVegaJMPA}, \cite{Lions}, \cite{LionsII}. The reader is encouraged to look at the expositions by Bouchut \cite{Bouchut} and Perthame \cite{PerthameBAMS} for greater detail concerning the motivation for averaging lemmas, for instance, to derive information about solutions of more elaborate kinetic equations.

In the present work, we are interested in velocity averages
\begin{equation*}
\rho f(x,t)=\int_{\mathbb{R}^d}f(x-tv,v) \,\mathrm{d}\mu(v),\quad (x,t) \in \mathbb{R}^{d+1},
\end{equation*}
of the solution $F(x,v,t) = f(x-tv,v)$ of the homogeneous initial value problem $(\partial _t +v\cdot\nabla)F=0$ with initial data $F(x,v,0)=f(x,v)$. Here the averaging is taken over the unit sphere or ball equipped with their usual induced Lebesgue measures, and we write $\rho = \rho_\textrm{s}$ or $\rho = \rho_\mathrm{b}$, respectively, for these averages. We shall see later that there is a natural way to unify the smoothing estimates we seek for these velocity domains, so for the sake of simplicity of the exposition, we focus this introductory discussion on the unit sphere. As proved in the work of Bournaveas and Perthame \cite{BP}, when $d=3$, a half derivative gain is achievable in the sense of classical Sobolev spaces. When $d=2$, they also observed such a gain is not possible and established a natural replacement through the use of so-called \emph{hyperbolic derivatives}. More precisely, it was shown in \cite{BP} that
\begin{equation} \label{e:smoothingL^2}
\| \diffp\diffn  \rho_\textrm{s} f\|_{L^2_{t,x}}\leq C\|f\|_{L^2_{x,v}}
\end{equation}
holds when $(d,\beta_+,\beta_-) = (3,\frac{1}{2},0)$ and $(d,\beta_+,\beta_-) = (2,\frac{1}{4},\frac{1}{4})$. Here, $D_+^{\beta_+}$ and $D_-^{\beta_-}$ are Fourier multiplier operators with respective multipliers $(|\xi| + |\tau|)^{\beta_+}$ and $||\xi| - |\tau||^{\beta_-}$, corresponding to classical fractional derivatives and hyperbolic derivatives. This was extended to general dimensions $d \geq 2$ with $(\beta_+,\beta_-) = (\frac{d-1}{4},\frac{3-d}{4})$ by Bournaveas and Guti\'errez in \cite{BGRevista}, and it turns out that this is sharp in the sense that \eqref{e:smoothingL^2} fails for $\beta_- < \frac{3-d}{4}$ (equivalently, \eqref{e:smoothingL^2} fails for $\beta_+ > \frac{d-1}{4}$, since the scaling condition $\beta_+ + \beta_- = \frac{1}{2}$ is necessary for \eqref{e:smoothingL^2}).

In very recent work \cite{BBGL}, the purely $L^2$-based results in \cite{BP} and \cite{BGRevista} were significantly extended to estimates of the form
\begin{equation} \label{e:smoothing}
\| \diffp\diffn  \rho f\|_{L^q_{t}L^r_x}\leq C\|f\|_{L^2_{x,v}}
\end{equation}
for general $q,r \in [2,\infty)$ and where $\rho$ is either $\rho_\mathrm{s}$ or $\rho_\mathrm{b}$. In this case, the scaling condition $\beta_+ + \beta_- = \frac{d}{r} + \frac{1}{q} - \frac{d}{2}$ is necessary for \eqref{e:smoothing}, and examples in \cite{BBGL} show that
$
\beta_- \geq \frac{1}{q} + \frac{d-1}{2r} - \frac{d-1}{2}$ 
and
$
\beta_- > \frac{1-d}{4}
$
are also necessary conditions. Furthermore, it was shown in \cite{BBGL} that \eqref{e:smoothing} holds if $\beta_- > \beta^*_-$, where
\[
\beta^*_- = \max\bigg\{ \frac{1}{q} + \frac{d-1}{2r} - \frac{d-1}{2}, -\frac{d-1}{4}\bigg\},
\]
which left open the question of whether \eqref{e:smoothing} holds in the critical case $\beta_- = \beta_-^*$ for $q, r \in [2,\infty)$ such that $\frac{1}{q} > \frac{d-1}{2}(\frac{1}{2} - \frac{1}{r})$. The purpose of the present paper is to resolve this issue and thus establish the following complete characterisation of the exponents for which \eqref{e:smoothing} holds. 
\begin{theorem}[Unit sphere] \label{t:mainsphere}
Suppose $d \geq 2$, $q,r \in [2,\infty)$ and $\beta_+ + \beta_- = \frac{d}{r} + \frac{1}{q} - \frac{d}{2}$. Let $\rho_\mathrm{s}$ be the velocity averaging operator given by
\[
\rho_\mathrm{s} f(x,t)=\int_{\mathbb{S}^{d-1}}f(x-tv,v) \,\mathrm{d}\sigma(v),
\]
where $\sigma$ denotes the usual induced Lebesgue measure on the unit sphere $\mathbb{S}^{d-1}$.
\begin{enumerate}
\item Suppose $\frac{1}{q} \leq \frac{d-1}{2}(\frac{1}{2} - \frac{1}{r})$. Then \eqref{e:smoothing} holds if and only if $\beta_- > \beta^*_-$.
\item Suppose $\frac{1}{q} > \frac{d-1}{2}(\frac{1}{2} - \frac{1}{r})$. Then \eqref{e:smoothing} holds if and only if $\beta_- \geq \beta^*_-$.
\end{enumerate}
\end{theorem}
For the unit ball, the analogous statement is the following.
\begin{theorem}[Unit ball] \label{t:mainball}
Suppose $d \geq 2$, $q,r \in [2,\infty)$ and $\beta_+ + \beta_- = \frac{d}{r} + \frac{1}{q} - \frac{d}{2}$. Let $\rho_\mathrm{b}$ be the velocity averaging operator given by
\[
\rho_\mathrm{b} f(x,t)=\int_{\mathbb{B}^{d-1}}f(x-tv,v) \,\mathrm{d}v.
\]
\begin{enumerate}
\item Suppose $\frac{1}{q} \leq \frac{d-1}{2}(\frac{1}{2} - \frac{1}{r})$. Then \eqref{e:smoothing} holds if and only if $\beta_- > \beta^*_-	-\frac{1}{2}$.
\item Suppose $\frac{1}{q} > \frac{d-1}{2}(\frac{1}{2} - \frac{1}{r})$. Then \eqref{e:smoothing} holds if and only if $\beta_- \geq \beta^*_--\frac{1}{2}$.
\end{enumerate}
\end{theorem}
We clarify that in the region $\frac{1}{q} > \frac{d-1}{2}(\frac{1}{2} - \frac{1}{r})$ where we make our contribution in the present paper, in the pure-norm case $q = r$, it has been observed elsewhere that the estimates \eqref{e:smoothing} hold in the critical case $\beta_- = \beta_-^*$. As already noted above, when $(q,r) = (2,2)$, this can be found in \cite{BP} in two spatial dimensions and \cite{BP} in general. Also, it was noted in \cite{BBGL} that \eqref{e:smoothing} holds in the critical case $\beta_- = \beta_-^*$ for general $q=r$ by utilising the cone multiplier estimates in \cite{Lee}. Such cone multiplier estimates ultimately relied upon the bilinear theory for the Fourier restriction problem. Here, our arguments are substantially different and rely on linear Fourier restriction theory combined with bilinear interpolation in the spirit of the Keel--Tao argument in \cite{KeelTao}; the advantage of the approach in the current paper is that it readily handles the mixed-norm case $q \neq r$. 

Our arguments in handling the critical estimates for the cone multiplier will be sufficiently flexible so that they may be applied in related contexts. As a tangible example, we also establish the corresponding estimates for the paraboloid at the end of the paper.

\begin{organisation}
In the subsequent section, we state a theorem which unifies Theorems \ref{t:mainsphere} and \ref{t:mainball} along with the equivalent mixed-norm estimates for the cone multiplier operator. In Section \ref{section:cone} we prove the required estimates on the cone multiplier operator in the critical case, and finally in Section \ref{section:schrodinger} we show how our methods yield analogous estimates for the paraboloid.
\end{organisation}

\section{A unified theorem and connection with the cone multiplier operator}

Our proof of the critical case in both Theorems \ref{t:mainsphere} and \ref{t:mainball} are implied by the \emph{same} $L^2_{t,x} \to L^q_tL^r_x$ estimates for the cone multiplier operator, and thus it is natural to first present the following unified theorem.
\begin{theorem}[Unified] \label{t:mainunified}
Suppose $d \geq 2$, $q,r \in [2,\infty)$ and $\beta_+ + \beta_- = \frac{d}{r} + \frac{1}{q} - \frac{d}{2}$. Let $\rho_\kappa$ be the velocity averaging operator given by
\[
\rho_\kappa f(x,t) = \frac{1}{\Gamma(1+\kappa)} \int_{\mathbb{B}^{d-1}}f(x-tv,v) \, (1-|v|^2)^{\kappa} \, \mathrm{d}v
\]
for $\kappa \in [-1,0]$. 
\begin{enumerate}
\item Suppose $\frac{1}{q} \leq \frac{d-1}{2}(\frac{1}{2} - \frac{1}{r})$. Then \eqref{e:smoothing} holds if and only if $\beta_- > \beta^*_-(\kappa)$.
\item Suppose $\frac{1}{q} > \frac{d-1}{2}(\frac{1}{2} - \frac{1}{r})$. Then \eqref{e:smoothing} holds if and only if $\beta_- \geq \beta^*_-(\kappa)$.
\end{enumerate}
Here, 
\[
\beta^*_-(\kappa) = \max\bigg\{ \frac{1}{q} + \frac{d-1}{2r} - \frac{d + \kappa}{2}, -\frac{d + 1 + 2\kappa}{4} \bigg\}.
\]
\end{theorem}
\begin{remark}
The distribution $ \frac{1}{\Gamma(1+\kappa)} (1-|v|^2)^{\kappa}_+$ is defined for $\kappa \leq -1$ via analytic continuation; when $\kappa = -1$ this distribution coincides with $\frac{1}{2}\mathrm{d}\sigma$, and therefore $\rho_{-1} = \frac{1}{2}\rho_\mathrm{s}$. Also, $\beta_-^*(-1)$ coincides with $\beta_-^*$ introduced in Section 1. Thus, Theorem 2.1 clearly contains both Theorem \ref{t:mainsphere} and Theorem \ref{t:mainball} by considering $\kappa = -1,0$.
\end{remark}
To prove Theorem \ref{t:mainunified}, we shall need the cone multiplier operator $\mathcal{C}^\alpha$ of order $\alpha$ defined by
\[
\mathcal{F}(\mathcal{C}^\alpha g)(\xi,\tau) = \bigg(1-\frac{\tau^2}{|\xi|^2}\bigg)^\alpha_+\phi(|\xi|) \widehat{g}(\xi,\tau).
\]
Here, $\phi \in C^\infty_c(\mathbb{R})$ is supported in $[\frac{1}{2},2]$ and $\alpha > -1$, and for appropriate functions $g : \mathbb{R}^{n} \to \mathbb{C}$, $n \geq 1$, we use the Fourier transform given by
\[
\widehat{g}(\xi) = \mathcal{F}g(\xi) = \int_{\mathbb{R}^n} e^{-ix\cdot\xi} g(x) \, \mathrm{d}x.
\]
Given the following result, this definition is natural for the purposes of the connection with estimates of the form \eqref{e:smoothing}, and we refer the reader to \cite{BBGL} for further discussion on how this operator mildly differs from the more standard cone multiplier.
\begin{theorem}[\cite{BBGL}]\label{t:coneequiv}
Suppose $d \geq 2$, $q,r\in [2,\infty)$, $\kappa \in [-1,0]$ and $\beta_+ + \beta_- = \frac{d}{r} + \frac{1}{q} - \frac{d}{2}$. Then 
\begin{equation} \label{e:gensmoothing}
\| \diffp \diffn  \rho_\kappa f\|_{L^q_tL_x^r}\leq C\|f\|_{L^2_{x,v}}
\end{equation}
holds if and only if the cone multiplier operator $\mathcal{C}^{\beta_- + \frac{\kappa}{2} + \frac{d-1}{4}}$ is bounded from $L^2_{t,x}$ to $L^q_tL^r_x$.
\end{theorem}
For details and a proof of Theorem \ref{t:coneequiv}, we refer the reader to \cite[Theorem 4.3]{BBGL}.

Given Theorem \ref{t:coneequiv}, it is clear that to prove Theorem \ref{t:mainunified} it suffices to prove the following.
\begin{theorem} \label{t:coneestimates}
Suppose $d \geq 2$ and $q,r\in [2,\infty)$.
\begin{enumerate}
\item Suppose $\frac{1}{q} \leq \frac{d-1}{2}(\frac{1}{2} - \frac{1}{r})$. Then $\mathcal{C}^{\alpha}$ is bounded from $L^2_{t,x}$ to $L^q_tL^r_x$ if and only if $\alpha > \alpha^*$.
\item Suppose $\frac{1}{q} > \frac{d-1}{2}(\frac{1}{2} - \frac{1}{r})$. Then $\mathcal{C}^{\alpha}$ is bounded from $L^2_{t,x}$ to $L^q_tL^r_x$ if and only if $\alpha \geq \alpha^*$.
\end{enumerate}
Here,
\begin{equation*}
\alpha^* = \alpha^*(q,r) = \max\bigg\{\frac{1}{q}+\frac{d-1}{2r}-\frac{d+1}{4}, -\frac{1}{2}\bigg\}.
\end{equation*}
\end{theorem}
In the subsequent section, we shall give a proof of the boundedness of $\mathcal{C}^{\alpha}$ from $L^2_{t,x}$ to $L^q_tL^r_x$ when $\alpha = \alpha^*$ and $\frac{1}{q} > \frac{d-1}{2}(\frac{1}{2} - \frac{1}{r})$; the remaining estimates can be found in \cite{BBGL}. We also refer the reader to \cite{BBGL} for further background and literature on the cone multiplier operator and discussion of its prominent role in contemporary euclidean harmonic analysis.

\section{Proof of Theorem \ref{t:coneestimates}: the critical case $\alpha = \alpha^*$} \label{section:cone}

We shall adopt the notation $A \lesssim B$ to signify the inequality $A \leq CB$, where $C$ is a constant which is permitted to depend only on $d$ and any Lebesgue space exponents which are under consideration, and $A \sim B$ means $A \lesssim B \lesssim A$.

Recall that we are only interested in region $\frac{1}{q} > \frac{d-1}{2}(\frac{1}{2} - \frac{1}{r})$, in which case
\[
\alpha^*(q,r) = \frac{1}{q}+\frac{d-1}{2r}-\frac{d+1}{4} = \frac{1}{q}-\frac{1}{2}-\frac{d-1}{2}\bigg(\frac{1}{2} - \frac{1}{r}\bigg) .
\]

\subsection{Outline} \label{subsection:outline}
Since the singularity in the multiplier of $\mathcal{C}^\alpha$ occurs on the conical surfaces $\tau = \pm |\xi|$, we begin in a standard manner by decomposing dyadically away from these surfaces. It will suffice to consider the upward conical region, and thus we introduce the multiplier operator $\mathcal{C}_{\delta},$ given by
\begin{equation} \label{e:Cdeltadefinition}
\mathcal{F}(\mathcal C_{\delta} g)(\xi,\tau) = \psi\left( \frac{|\xi|-\tau}{\delta}\right) \phi(|\xi|)  \widehat{g}(\xi,\tau)
\end{equation}
for $0 < \delta \ll 1$. Here, $\psi \in C^\infty_c(\mathbb{R})$ has support in $[\frac{1}{2}, 2]$ and satisfies
\begin{equation} \label{e:monodecomp}
s^{\alpha}=\sum_{k \in \mathbb{Z}} 2^{- k\alpha} \psi(2^{k} s)
\end{equation}
for all $s > 0$ (the existence of such a bump function is easily justified via the standard smooth partition of unity).

For $d\geq3,$ our proof proceeds by considering separately the cases $r=2$ and $q=2$ in the relevant region $\frac{1}{q} > \frac{d-1}{2}(\frac{1}{2} - \frac{1}{r})$; complex interpolation then gives all desired estimates in this region. The case $r=2$ is easier so we begin in this case. Here, we show that estimates of the form
\begin{equation*}
\|\C_{\delta}g\|_{L^q_tL^2_x}\lesssim \delta^{-(\frac{1}{q}-\frac{1}{s})}\|g\|_{L^s_tL^2_x},
\end{equation*}
hold for $q\geq s\geq2$, from which an interpolation argument yields weak-type estimates for $\C^\alpha$, where $\alpha=\alpha^*(q,2),$ and then real interpolation gives the desired strong-type estimates. For $q=2,$ we use a family of estimates of the form 
\begin{equation*}
|\agl{\C_{\delta}f,g}|	\lesssim \delta^{\frac{d-1}{2}{(\frac{1}{2}-\frac{1}{r'})}+\frac{d-1}{2}{(\frac{1}{2}-\frac{1}{s'})}}\|f\|_{L^2_tL^{r}_x}\|g\|_{L^2_tL^{s}_x},
\end{equation*}
for appropriate $r$ and $s$, combined with bilinear real interpolation; this argument was inspired by the Keel-Tao proof of the endpoint Strichartz estimates \cite{KeelTao}. When $d=2,$ the above argument fails to generate the full region $\frac{1}{q} > \frac{d-1}{2}(\frac{1}{2} - \frac{1}{r})$. To overcome this, additionally, we run a similar  bilinear argument for the case $q=4$ (and $r\geq4$); by complex interpolation with the case $q=2$ and $r=2$, we are then able to get all desired estimates.
\clearpage
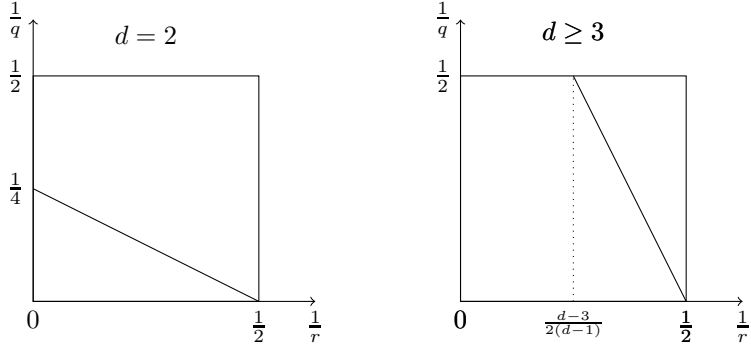
\begin{figure} \label{fig:waveadmissible}
\begin{center}
\begin{tikzpicture}[scale=3]
    % Draw axes
    \draw [<->] (0,2.5/2) node (yaxis) [left] {$\frac{1}{q}$}
        |- (2.5/2,0) node (xaxis) [below] {$\frac{1}{r}$};
        \draw (0, 0) rectangle (1, 1);
% %labels
\node [left]at (0,1) {$\frac{1}{2}$};
\node [left]at (0,1/2) {$\frac{1}{4}$};
\node[below] at (0,0) {$0$};
\node[below] at (1,0) {$\frac{1}{2}$};
\node[above] at (1/2,1.1) {$d=2$};
\draw  (0,.5)--(1,0);
\end{tikzpicture}
\hspace{10mm}
\begin{tikzpicture}[scale=3]
    % Draw axes
    \draw [<->] (0,2.5/2) node (yaxis) [left] {$\frac{1}{q}$}
        |- (2.5/2,0) node (xaxis) [below] {$\frac{1}{r}$};
        \draw (0, 0) rectangle (1, 1);
% %labels
\node [left]at (0,1) {$\frac{1}{2}$};

\node[below] at (0,0) {$0$};
\node[below] at (1,0) {$\frac{1}{2}$};
\node[above] at (1/2,1.1) {$d\geq 3$};
\node[below] at (1/2,0) {\tiny $\frac{d-3}{2(d-1)}$};
\node[below] at (0,0) {0};
\node[below] at (1,0) {$\frac{1}{2}$};
\node[above] at (1/2,1.1) {$d\geq 3$};
\draw[dotted] (1/2,0)--(1/2,1);
\draw  (1,0)--(.5,1);
\end{tikzpicture}
\end{center}
\caption{The line $\frac{1}{q} = \frac{d-1}{2}(\frac{1}{2} - \frac{1}{r})$}
\end{figure}

\subsection{Dyadic decomposition} Initially we follow the argument in \cite{BBGL}, first by observing that the desired estimate
\begin{equation} \label{e:Calphagoal}
\|\mathcal{C}^\alpha g\|_{L^q_tL^r_x} \lesssim \|g\|_{L^2}
\end{equation}
is implied by the boundedness from $L^2$ to $L^q_tL^r_x$ of the multiplier operator with multiplier $(|\xi| - \tau)_+^\alpha\phi(|\xi|)$ (our goal is $\alpha = \alpha^*(q,r)$, but in this series of reductions, $\alpha$ may be general). By using \eqref{e:monodecomp} to make a dyadic decomposition of this multiplier, it suffices to prove
\begin{equation} \label{e:afterreductions}
\bigg\|\sum_{k=k_0}^\infty 2^{- k\alpha} \mathcal{C}_k g\bigg\|_{L^q_tL^r_x} \lesssim \|g\|_{L^2},
\end{equation}
where we abbreviate $\mathcal{C}_{2^{-k}}$ to $\mathcal{C}_k$, and for an appropriately large choice of $k_0 \sim 1$ (to be determined later). Our proof of \eqref{e:afterreductions} will crucially rely on sharp estimates for each operator $\mathcal{C}_\delta$. 
\begin{proposition}\label{p:kdeltaest} 
Let $\tilde \psi\in\mathcal S(\R)$ and $\phi\in C_c^{\infty}(\R)$ supported in $[\frac{1}{2},2]$. Let $0<\delta\ll1$ and let $\tilde K_{\delta}$ be defined by
\[
\widehat{\tilde K_\delta}(\xi,\tau) = \phi(|\xi|)\tilde\psi\left(\frac{|\xi|-\tau}{\delta}\right) .
\]
Then
\begin{equation}\label{e:deltaest}
\|\tilde K_{\delta}\ast g\|_{L^q_tL^r_x}\lesssim \delta^{-\alpha^*(q,r)}\|g\|_{L^2}
\end{equation}
whenever $d\geq2, q,r\in[2,\infty],\frac{1}{q} \geq \frac{d-1}{2}(\frac{1}{2} - \frac{1}{r})$ and $(q,r,d) \neq (2,\infty,3)$. 
\end{proposition}
This proposition is clearly applicable to the operator $\C_\delta$ by taking $\tilde\psi$ to have compact support; we shall need the slightly more general version stated above in order to prove (\ref{sqest}.). Although Proposition \ref{p:kdeltaest} was proved in \cite{BBGL}, we include some details below for reasons that will become apparent later.
\begin{proof}[Proof of Proposition \ref{p:kdeltaest}] Firstly, when $q = r =2$, the claimed estimate \eqref{e:deltaest} follows immediately from Plancherel's theorem. Thus, for $d \geq 4$, by interpolation, it is enough to check
\begin{equation} \label{e:sqrtdeltaboundg}
\|\tilde K_{\delta}\ast g\|_\lqr \lesssim  \delta^{1/2} \|g\|_{L^2}
\end{equation}
for $q$ and $r$ such that $\frac{1}{q} =  \frac{d-1}{2}(\frac12-\frac1r)$. To see this, we use the Fourier inversion formula and a change of variables (translation in time frequency variables) to write
\[
\tilde K_{\delta}\ast g(x,t) = c_d\int_\mathbb{R}  e^{-i st}\tilde\psi(\delta^{-1} s) e^{it\sqrt{-\Delta}}h_s(x) \, \mathrm{d}s
\]
where $\widehat{h_s}(\xi) = \phi(|\xi|)\widehat{g}(\xi, |\xi|-s)$. Applying the classical Strichartz estimates for the wave propagator (which are applicable since $\frac{1}{q} =  \frac{d-1}{2}(\frac12-\frac1r)$), we obtain
\[
\|\tilde K_{\delta}\ast g\|_\lqr \lesssim \int_\mathbb{R}  |\tilde\psi(\delta^{-1} s)| \bigg(\int_{\mathbb{R}^d} |\widehat{g}(\xi, |\xi|-s)|^2 \, \mathrm{d}\xi\bigg)^{1/2} \mathrm{d}s
\]
and now an application of the Cauchy--Schwarz inequality yields \eqref{e:sqrtdeltaboundg}; this completes the proof of Proposition \ref{p:kdeltaest} for $d \geq 4$.

For $d=2,3$ further arguments are required. For $d=2$, this is because the endpoint Strichartz estimate occurs at $(q,r) = (4,\infty)$, and for $d=3$, this is because of the absence of a Strichartz estimate for the wave propagator in the case $(d,q,r) = (3,2,\infty)$. The details of these arguments will not be of particular benefit for the current paper, so we simply refer the reader to \cite[Section 5]{BBGL}.
\end{proof}
\begin{remark}
Below, in the case $r=2$, it will be necessary to control the dependence on the implicit constant in \eqref{e:deltaest} in terms of the function $\tilde{\psi}$. It is clear from the above proof that we have
\begin{equation}\label{e:deltaestpsi}
\|\tilde K_{\delta}\ast g\|_{L^q_tL^2_x}\leq C\|\tilde{\psi}\|_\infty^{\frac{2}{q}} \|\tilde{\psi}\|_2^{1-\frac{2}{q}} \delta^{-(\frac{1}{q} - \frac{1}{2})}\|g\|_{L^2},
\end{equation}
where $C$ depends on $d$, $q$ and $\|\phi\|_\infty$. Moreover, this argument yields \eqref{e:deltaestpsi} for all $d \geq 2$.
\end{remark}

%%%%%%%%%%%%%%%%%%%%%%%%%%%%%%%
\subsection{The case $r=2$} \label{subsection:r2}
%%%%%%%%%%%%%%%%%%%%%%%%%%%%%%%
Our goal in this subsection is to prove \eqref{e:afterreductions} when $r=2$ at the critical exponent $\alpha=\alpha^*(q,2)=\frac{1}{q}-\frac{1}{2}$; that is,
\begin{equation}  \label{e:r2}
\bigg\|\sum_{k=k_0}^\infty 2^{- k\alpha} \mathcal{C}_k g\bigg\|_{L^q_tL^2_x} \lesssim \|g\|_{L^2_tL^2_x},
\end{equation}
where $\alpha = \frac{1}{q}-\frac{1}{2}$. The main ingredient in the proof are the estimates
\begin{equation}\label{e:s2}
\|\C_{\delta}g\|_{L^q_tL^2_x}\lesssim \delta^{-(\frac{1}{q}-\frac{1}{s})}\|g\|_{L^s_tL^2_x}, \quad q\geq s\geq2.
\end{equation}
We remark that using Proposition \ref{p:kdeltaest} and duality gives such an estimate for $1\leq s \leq 2 \leq q \leq\infty$; however, this range of $s$ and $q$ appears to be insufficient to obtain (\ref{e:r2}). To prove (\ref{e:s2}) for all $2\leq s\leq q$ we use a localisation argument in the temporal variable which has been used several times in literature on related multiplier estimates. If $\widehat{\psi}$ had compact support, (\ref{e:s2}) would follow easily via standard arguments for localised operators and H\"older's inequality. Since $\psi$ itself has compact support, a little additional work is required to handle $\widehat{\psi}$ in the Schwartz class.
\begin{proof}[Proof of \eqref{e:s2}]
Recall that $\psi\in C_c^{\infty}(\R)$ and $\phi\in C_c^{\infty}(\R)$ are both supported in $[\frac{1}{2},2]$, and let $K_\delta$ be given by
$$
\widehat{K}_\delta (\xi,\tau)=\psi\bigg(\frac{|\xi|-\tau}{\delta}\bigg)\phi(|\xi|).
$$ 
Next, by a standard smooth partition of unity, we write 
\begin{equation}\label{Keq}
K_\delta(x,t)=K_{\delta,0}(x,t)+\sum_{j \geq 1}K_{\delta,j}(x,t),
\end{equation}
where $K_{\delta,0}(x,t)=K_{\delta}(x,t)\eta_0(t\delta)$ and $K_{\delta,j}(x,t)=K_{\delta}(x,t)\eta(t\delta/2^j)$. Here,  $\eta_0 \in C_c^{\infty}(\R)$ is supported in a $O(1)$ neighbourhood of the origin, and $\eta \in C_c^\infty(\mathbb{R})$ is supported in $[\frac{1}{2},2]$, such that $\eta_0(s)+\sum_{j \geq 1}\eta(s/2^j)=1$ for $s > 0$. 

First we claim that
\begin{equation} \label{e:K0}
\|K_{\delta,0}\ast g\|_{L^q_tL^2_x}\lesssim \delta^{-(\frac{1}{q}-\frac{1}{s})}\|g\|_{L^s_tL^2_x}
\end{equation}
Since $K_{\delta,0}(x,\cdot)$ is supported in a $O(1/\delta)$ neighbourhood of the origin, by standard localisation arguments, we may assume that that $g$ is compactly supported in the temporal variable in some interval of length $O(1/\delta)$. For such $g$, we now apply Proposition \ref{p:kdeltaest} with $\tilde{\psi} = \psi * \widehat{\eta_0}$, followed by H\"older's inequality to obtain \eqref{e:K0}.

In a similar way, in estimating the contribution from $K_{\delta,j} * g$, we may assume that $g$ is compactly support in the temporal variable in some interval of length $O(2^j/\delta)$. From Proposition \ref{p:kdeltaest} with $\tilde{\psi} = \psi * \widehat{\eta}(2^j \cdot)$ (more precisely, using \eqref{e:deltaestpsi}), we obtain
\[
\|K_{\delta,j}\ast g\|_{L^q_tL^2_x}\lesssim_N \frac{1}{2^{jN}} \delta^{-(\frac{1}{q}-\frac{1}{2})}\|g\|_{L^2_tL^2_x}
\]
for any $N \geq 1$, since the support restriction on $\eta$ and Plancherel's theorem yields $\|\psi * \widehat{\eta}(2^j\cdot)\|_2 \lesssim_N 2^{-jN}$. Again using H\"older's inequality in the temporal variable, and then summing a convergent geometric series, we obtain
\[
\sum_{j \geq 1} \|K_{\delta,j}\ast g\|_{L^q_tL^2_x}\lesssim  \delta^{-(\frac{1}{q}-\frac{1}{s})}\|g\|_{L^s_tL^2_x}
\]
and hence \eqref{e:s2}.
\end{proof}

\subsection*{(\ref{e:s2}) implies (\ref{e:r2}).} 

We shall prove the family of estimates
\begin{equation}\label{e:r2strong}
\bigg\|\sum_{k=k_0}^\infty 2^{- k\alpha} \mathcal{C}_k g\bigg\|_{L^q_tL^2_x} \lesssim \|g\|_{L^{s}_tL^2_x}
\end{equation}
where $\alpha = \frac{1}{q} - \frac{1}{s}$, and $(\frac{1}{s},\frac{1}{q})$ lies in the interior of the convex hull of $(0,0)$, $(\frac{1}{2},0)$, $(1,\frac{1}{2})$ and $(1,1)$; see Figure 2. Clearly this includes the desired estimates in \eqref{e:r2} by setting $s=2$.

To see \eqref{e:r2strong}, we claim that \eqref{e:s2} implies the restricted weak type estimates
\begin{equation}\label{e:RWT}
\bigg\|\sum_{k=k_0}^\infty 2^{- k\alpha} \mathcal{C}_k g\bigg\|_{L^{q,\infty}_tL^2_x}\lesssim \|g\|_{L^{s,1}_tL^2_x}
\end{equation}
whenever $2\leq s\leq q<\infty$ and $\alpha = \frac{1}{q} - \frac{1}{s}$ (corresponding to $(\frac{1}{s},\frac{1}{q})$ in the triangle with vertices $(0,0)$, $(\frac{1}{2},0)$ and $(\frac{1}{2},\frac{1}{2})$ with the bottom edge omitted). Prior to proving this claim, we show how it yields \eqref{e:r2strong}. Indeed, by duality we immediately obtain the estimates
\begin{equation}\label{e:RWTdual}
\bigg\|\sum_{k=k_0}^\infty 2^{- k\alpha} \mathcal{C}_k g\bigg\|_{L^{q,\infty}_tL^2_x} \lesssim \|g\|_{L^{s,1}_tL^2_x}
\end{equation}
whenever $1 < s \leq q \leq 2$ and $\alpha = \frac{1}{q} - \frac{1}{s}$ (corresponding to $(\frac{1}{s},\frac{1}{q})$ in the triangle with vertices $(\frac{1}{2},\frac{1}{2})$, $(1,\frac{1}{2})$ and $(1,1)$ with the right edge omitted). Finally, real interpolation between \eqref{e:RWT} and \eqref{e:RWTdual} gives \eqref{e:r2strong} (here, we consider $\alpha$ as fixed and interpolate along $(\frac{1}{s},\frac{1}{q})$ line segments satisfying $\alpha = \frac{1}{q} - \frac{1}{s}$).
\begin{figure}
\begin{center}
\begin{tikzpicture}[scale=6]
    % Draw axes
    \draw [<->] (0,2.5/2) node (yaxis) [left] {$\frac{1}{q}$}
        |- (2.5/2,0) node (xaxis) [below] {$\frac{1}{s}$};
        \draw (0, 0) rectangle (1, 1);
% %labels
\node [left]at (0,1) {1};
\node [left]at (0,1/2) {$\frac{1}{2}$};
\node[below] at (1/2,0) {$\frac{1}{2}$};
\node[below] at (0,0) {0};
\node[below] at (1,0) {1};
% % % %

%\shadedraw (0,0)  -- (1/2,0) --(1/2,1/2);
%\shadedraw (1/2,1/2)  -- (1,1) --(1,1/2);

\shadedraw  (.5,0)--(1,.5)--(1,1)--(0,0);
\draw (0,1/2)--(1/2,1/2)--(1,1/2);

\draw (.7/2, .1)-- (.7/2, .3);

\fill(.7/2, .2/2) circle [radius=.25pt]  (.7/2, .45/2)circle [radius=.25pt] (.7/2, .3)circle [radius=.25pt];

\node [left]at (.7/2, .2/2) {\tiny ($1/s$,\tiny $1/q_2$)};
\node [right]at (.7/2, .45/2) {\tiny ($1/s$,\tiny $1/q$)};
\node [right]at (.7/2, .7/2) {\tiny ($1/s$,\tiny $1/q_1$)};

\draw (0,0)--(1/2,0);
\end{tikzpicture}
\end{center}
\caption{}
\end{figure}

We now justify the remaining claim.
\begin{proof}[Proof of (\ref{e:RWT})]
Fix $q,s$ such that $2 \leq s < q < \infty$ and define $\alpha =\frac{1}{q}-\frac{1}{s}$. We will prove the slightly stronger estimate
\begin{equation} \label{e:WT}
\bigg\|\sum_{k=k_0}^\infty 2^{- k\alpha} \mathcal{C}_k g\bigg\|_{L^{q,\infty}_tL^2_x} \lesssim \|g\|_{L^s_tL^2_x}.
\end{equation}
Let $\lambda > 0$ and assume, without loss of generality, that $\|g\|_{L^s_tL^2_x} = 1$. Then, by Chebyshev's inequality and the triangle inequality, we have
\[
\bigg|\bigg\{t: \bigg\|\sum_{k=k_0}^\infty 2^{- k\alpha} \mathcal{C}_k g\bigg\|_{L^2_x} >\lambda\bigg\}\bigg| \lesssim \lambda^{-q_1} I_1 + \lambda^{-q_2} I_2
\]
where
\begin{align*}
I_1 & := \int_\mathbb{R} \bigg(\sum_{k=k_0}^N 2^{-k\alpha}\|\mathcal{C}_{k}g(\cdot,t)\|_{L^2_x}\bigg)^{q_1}\, \mathrm{d}t \\
I_2 & := \int_\mathbb{R} \bigg(\sum_{k=N+1}^\infty 2^{-k\alpha}\|\mathcal{C}_{k}g(\cdot,t)\|_{L^2_x}\bigg)^{q_2}\, \mathrm{d}t,
\end{align*}
$N\in\mathbb{Z}$ will be chosen at the end of the proof, and $q_1$ and $q_2$ are such that $2 \leq s \leq q_1 < q < q_2 < \infty$ (see Figure 2). 

By the triangle inequality and \eqref{e:s2}, we obtain
$
I_1^{1/q_1} \lesssim 2^{N(\frac{1}{q_1} - \frac{1}{q})} 
$
and
$
I_2^{1/q_2} \lesssim 2^{-N(\frac{1}{q} - \frac{1}{q_2})}.
$
Putting these bounds together and optimising in the choice of $N$ yields
\[
\bigg|\bigg\{t: \bigg\|\sum_{k=k_0}^\infty 2^{- k\alpha} \mathcal{C}_k g\bigg\|_{L^2_x} >\lambda\bigg\}\bigg| \lesssim \lambda^{-q} 
\]
as claimed.
\end{proof}

%%%%%%%%%%%%%%%%%%%%%%%%%%%%%%%
\subsection{The case $q=2$} \label{subsection:q2}
%%%%%%%%%%%%%%%%%%%%%%%%%%%%%%%

Our next goal is to prove
\begin{equation}\label{e:q2}
\bigg\|\sum_{k=k_0}^\infty 2^{- k\alpha} \mathcal{C}_k g\bigg\|_{L^2_tL^r_x} \lesssim \|g\|_{L^2_tL^2_x}
\end{equation}
at the critical exponent $\alpha=\alpha^*(2,r)=\frac{d-1}{2}{(\frac{1}{r}-\frac{1}{2})}$. Here $r \in (2,\frac{2(d-1)}{d-3})$ for $d \geq 4$, and $r \in (2,\infty)$ for $d=2,3$. Our argument uses bilinear real interpolation and is based on the estimates
\begin{equation}\label{e:Cdeltabilinear}
|\agl{\C_{\delta}f,g}|\lesssim \delta^{\frac{d-1}{2}{(\frac{1}{2}-\frac{1}{a})}+\frac{d-1}{2}{(\frac{1}{2}-\frac{1}{b})}}\|f\|_{L^2_tL^{a'}_x}\|g\|_{L^2_tL^{b'}_x}
\end{equation}
for all $a$ and $b$ such that $(\frac{1}{a},\frac{1}{b}) \in \mathfrak{S}_d$. Here, $\mathfrak{S}_d = [\frac{d-3}{2(d-1)},\frac{1}{2}]^2$ for $d \geq 4$, $\mathfrak{S}_3 = (0,\frac{1}{2}]^2$ and $\mathfrak{S}_2 = [0,\frac{1}{2}]^2$. To see this, one writes $\agl{\C_{\delta}f,g}$ as $\agl{\C_{\delta}f,\tilde{\C}_\delta g}$, where $\tilde{\C}_\delta$ is defined in the same way as $\mathcal{C}_\delta$, except that the corresponding bump functions $\tilde{\psi}$ and $\tilde{\phi}$ have slighter enlarged support and are equal to 1 on the support of $\psi$ and $\phi$. Then, \eqref{e:Cdeltabilinear} follows from the Cauchy--Schwarz inequality and \eqref{e:deltaest}.

The bilinear interpolation result we use is stated below (see, for example, \cite{BerghLofstrom}).
\begin{lemma}\label{l:LP}
Suppose $A_0,A_1,B_0,B_1,C_0,C_1$ are Banach spaces. Suppose also that the bilinear operator $T$ is bounded as follows:
\begin{equation*}
	\begin{aligned}
		T	&:A_0\times B_0\rightarrow C_0,\\
		T	&:A_0\times B_1\rightarrow C_1,\\
		T	&:A_1\times B_0\rightarrow C_1.
	\end{aligned}
\end{equation*}
Then, whenever $0<\theta_0,\theta_1<\theta<1$ and $p_0,p_1,\sigma \in [1,\infty]$ satisfy $1\leq \frac{1}{p_0}+\frac{1}{p_1}$ and $\theta_0+\theta_1=\theta,$ 
then
$$
T:	\left(A_0,A_1\right)_{\theta_0,\sigma p_0}	\times	\left(B_0,B_1\right)_{\theta_1,\sigma p_1}\rightarrow	\left(C_0,C_1\right)_{\theta,\sigma}
$$
is bounded.
\end{lemma}
%%%%%%%%%%%%%%%%%%%%%%%%%%%%%%proof of q=2%%%%%%%%%%%%%%%%%%%%%%%%%%
\begin{proof}[Proof of \eqref{e:q2}]
By duality, \eqref{e:q2} is equivalent to the estimate 
$$
\bigg\|\sum_{k=k_0}^\infty 2^{- k\alpha} \mathcal{C}_k g\bigg\|_{L^2_tL^2_x} \lesssim \|g\|_{L^2_tL^{r'}_x}
$$
for which it suffices to prove
\begin{equation} \label{e:sumq2}
\sum_{k=k_0}^\infty 2^{-2k\alpha} |\langle \mathcal{C}_k f, g \rangle|  \lesssim \|f\|_{L^{2}_tL^{r'}_x}\|g\|_{L^2_tL^{r'}_x}.
\end{equation}
We recall that $\alpha=\frac{d-1}{2}{(\frac{1}{r}-\frac{1}{2})}$.

If we define $\beta(a,b):=\frac{d-1}{2}{(\frac{1}{a}-\frac{1}{2})}+\frac{d-1}{2}{(\frac{1}{b}-\frac{1}{2})}$, then we can rewrite estimate (\ref{e:Cdeltabilinear}) as
\begin{equation}\label{e:T}
T:L^2_tL^{a'}_x	\times	L^2_tL^{b'}_x	\rightarrow	\ell^{\infty}_{\beta(a,b)}
\end{equation}
for all $(\frac{1}{a},\frac{1}{b}) \in \mathfrak{S}_d$, where $T=\{\agl{\C_{k}\cdot,\cdot }\}_{k}$ is the vector-valued bilinear operator corresponding to $\C_{k}$ and $\ell^p_\beta$ is the weighted sequence space with norm
\[
\| \{a_k\}_k\|_{\ell^p_\beta} = \bigg( \sum_k 2^{k\beta} |a_k|^p \bigg)^{1/p} \qquad (1 \leq p < \infty)
\]
and
$
\| \{a_k\}_k\|_{\ell^\infty_\beta} = \sup_k 2^{k\beta} |a_k|.
$

Now take general exponents $a_0,a_1,b_0,b_1$ such that $(\frac{1}{a_0},\frac{1}{b_0}), (\frac{1}{a_0},\frac{1}{b_1}), (\frac{1}{a_1},\frac{1}{b_0}) \in \mathfrak{S}_d$ and
$$		
\beta(a_0,b_1)=\beta(a_1,b_0)\neq\beta(a_0,b_0).
$$
The relation above is true as long as $\frac{1}{a_1}-\frac{1}{a_0} = \frac{1}{b_1}-\frac{1}{b_0}$ and $a_0 \neq a_1$; see Figure 3. From \eqref{e:Cdeltabilinear} we have that $T$ is bounded as follows:
	\begin{equation*}
	\begin{aligned}
		T	&:L^2_tL^{a_0'}_x	\times	L^2_tL^{b_0'}_x	\rightarrow	\ell^{\infty}_{\beta(a_0,b_0)},\\
		T	&:L^2_tL^{a_0'}_x	\times	L^2_tL^{b_1'}_x	\rightarrow	\ell^{\infty}_{\beta(a_0,b_1)},\\
		T	&:L^2_tL^{a_1'}_x	\times	L^2_tL^{b_0'}_x	\rightarrow	\ell^{\infty}_{\beta(a_1,b_0)}.
	\end{aligned}
\end{equation*}	
By Lemma \ref{l:LP} (with $\sigma=1, p_0=p_1=2$) we deduce that 
\begin{equation} \label{e:Tbounds}
T:	L^2_tL^{a',2}_x \times L^2_tL^{b',2}_x \rightarrow \ell^1_{\beta(a,b)}
\end{equation}
is bounded, where
\[
\frac{1}{a} = \frac{1-\theta_0}{a_0} + \frac{\theta_0}{a_1}, \quad \frac{1}{b} = \frac{1-\theta_1}{b_0} + \frac{\theta_1}{b_1}, \quad  0<\theta_0,\theta_1<\theta<1, \quad \theta_0 + \theta_1 = \theta.
\]
Also, for $\vartheta \in (0,1)$, we have used the interpolation identities
\begin{equation*}
(L^2_tL_x^{p_0},L^2_tL^{p_1}_x)_{\vartheta,2} =	L^2_tL_x^{p_\vartheta,2}
\end{equation*}
where $\frac{1}{p_\vartheta} = \frac{1-\vartheta}{p_0} + \frac{\vartheta}{p_1}$ (see, for example, \cite{LionsPeetre} or \cite{Cwikel}), and
\begin{equation*}
(\ell^\infty_{\beta_0},\ell^\infty_{\beta_1})_{\vartheta,1} = \ell^1_{(1-\vartheta)\beta_0 + \vartheta\beta_1}
\end{equation*}
where $\beta_0 \neq \beta_1$ (see, for example, \cite{BerghLofstrom}). In other words, we have the boundedness of $T$ as in \eqref{e:Tbounds} for all $(\frac{1}{a},\frac{1}{b})$ in the interior of the triangle with vertices $(\frac{1}{a_0},\frac{1}{b_0}), (\frac{1}{a_0},\frac{1}{b_1}), (\frac{1}{a_1},\frac{1}{b_0}) \in \mathfrak{S}_d$, and since these were arbitrary points of $\mathfrak{S}_d$, the boundedness holds in the interior of $\mathfrak{S}_d$. 

Finally, take $a=b=r$, where $r \in (2,\frac{2(d-1)}{d-3})$ for $d \geq 4$, or $r \in (2,\infty)$ for $d=2,3$. Since we have the inclusion $L^{r'} \subset L^{r',2}$, we obtain
$$
T:L^2_tL^{r'}_x	\times	L^2_tL^{r'}_x	\rightarrow	\ell^{1}_{(d-1){(\frac{1}{r}-\frac{1}{2})}},
$$
which implies (\ref{e:sumq2}), as desired.
\end{proof}

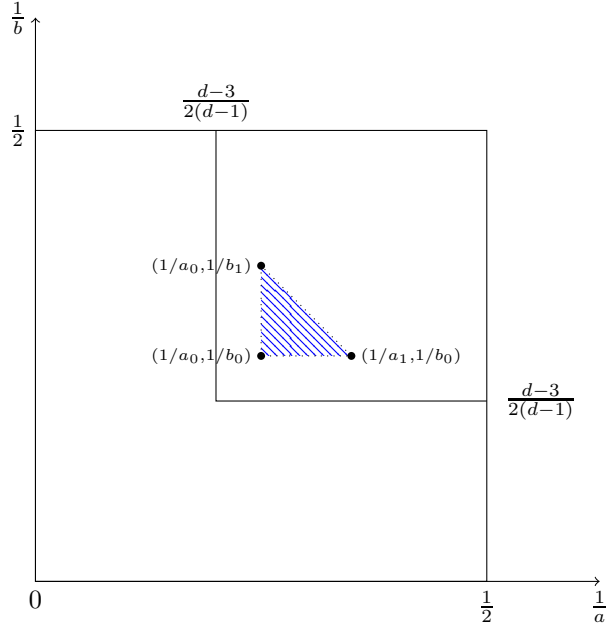
\begin{figure}
\begin{center}
\begin{tikzpicture}[scale=6]
    % Draw axes
    \draw [<->] (0,2.5/2) node (yaxis) [left] {$\frac{1}{b}$}
        |- (2.5/2,0) node (xaxis) [below] {$\frac{1}{a}$};
        \draw (0, 0) rectangle (1, 1);
% %labels
\node [left]at (0,1) {$\frac{1}{2}$};
\node [right]at (1,.4) {\ $\frac{d-3}{2(d-1)}$};
\node[above] at (.4,1) { $\frac{d-3}{2(d-1)}$};
\node[below] at (0,0) {0};
\node[below] at (1,0) {$\frac{1}{2}$};
% % % %

\draw (.4,1)--(.4,.4)--(1,.4);
\draw[pattern=north west lines, pattern color=blue,dotted] (.5,.5)--(.5,.7)--(.7,.5)--(.5,.5);

\fill(.5, .5) circle [radius=.25pt]  (.5, .7)circle [radius=.25pt] (.7, .5)circle [radius=.25pt];

\node [left]at (.5, .5) {\tiny ($1/a_0$,\tiny $1/b_0$)};
\node [left]at (.5, .7) {\tiny ($1/a_0$,\tiny $1/b_1$)};
\node [right]at (.7, .5) {\tiny ($1/a_1$,\tiny $1/b_0$)};
\end{tikzpicture}
\end{center}
\caption{The subsquare is the region $\mathfrak{S}_d$ and the bilinear interpolation step yields estimates in the shaded region}
\end{figure}

%%%%%%%%%%%%%%%%%%%%%%%%%%%%%%%
\subsection{The case $d=2$ and $q=4$} \label{subsection:q4}
%%%%%%%%%%%%%%%%%%%%%%%%%%%%%%%

In this section, $d=2$ and we shall prove \eqref{e:afterreductions} when $q=4$ and $r \in [4,\infty)$ at the critical exponent; that is, 
\begin{equation}  \label{e:q4}
\bigg\|\sum_{k=k_0}^\infty 2^{- k\alpha} \mathcal{C}_k g\bigg\|_{L^4_tL^r_x} \lesssim \|g\|_{L^2_tL^2_x},
\end{equation}
where $\alpha=\alpha^*(4,r)=\frac{1}{2}(\frac{1}{r}-1)$. Interpolating these estimates with \eqref{e:r2} and \eqref{e:q2} we obtain the desired estimates \eqref{e:afterreductions} whenever $q,r \in [2,\infty)$ are such that $\frac{1}{q} > \frac{1}{2}(\frac{1}{2} - \frac{1}{r})$.

\begin{proof}[Proof of \eqref{e:q4}]
We follow the same strategy used to prove \eqref{e:q2}, first noting that it suffices to prove
\begin{equation} \label{e:sumq4}
\sum_{k=k_0}^\infty 2^{-2k\alpha} |\langle \mathcal{C}_k f, g \rangle|  \lesssim \|f\|_{L^{\frac{4}{3}}_tL^{r'}_x}\|g\|_{L^{\frac{4}{3}}_tL^{r'}_x}
\end{equation}
for $r \in [4,\infty)$ and $\alpha=\frac{1}{2}(\frac{1}{r}-1)$. In this case we use the bilinear estimates
\begin{equation}\label{e:Cdeltabilinearq4}
|\agl{\C_{\delta}f,g}|	\lesssim \delta^{\frac{1}{2}{(\frac{1}{2}-\frac{1}{a})}+\frac{1}{2}{(\frac{1}{2}-\frac{1}{b})}+\frac{1}{2}}\|f\|_{L^\frac{4}{3}_tL^{a'}_x}\|g\|_{L^\frac{4}{3}_tL^{b'}_x},
\end{equation}
which are valid for all $(a,b) \in [2,\infty]^2$ (follow from \eqref{e:deltaest}), which may be interpreted as the boundedness of $T=\{\agl{\C_{k}\cdot,\cdot }\}_{k}$ from
\begin{equation}\label{e:T1}
T:L^\frac{4}{3}_tL^{a'}_x	\times	L^\frac{4}{3}_tL^{b'}_x	\rightarrow	\ell^{\infty}_{\gamma(a,b)},
\end{equation}
where $\gamma(a,b):=\frac{1}{2}{(\frac{1}{a}-\frac{1}{2})}+\frac{1}{2}{(\frac{1}{b}-\frac{1}{2})}-\frac{1}{2}$.

Suppose $(\frac{1}{a_0},\frac{1}{b_0}), (\frac{1}{a_0},\frac{1}{b_1}), (\frac{1}{a_1},\frac{1}{b_0}) \in \mathfrak{S}_2$ are such that $\frac{1}{a_1}-\frac{1}{a_0} = \frac{1}{b_1}-\frac{1}{b_0}$ and $a_0 \neq a_1$, and thus
$$		
\gamma(a_0,b_1)=\gamma(a_1,b_0)\neq\gamma(a_0,b_0).
$$
Then \eqref{e:Cdeltabilinearq4} implies that $T$ is bounded from
	\begin{equation*}
	\begin{aligned}
		T	&:L^\frac{4}{3}_tL^{a_0'}_x	\times	L^\frac{4}{3}_tL^{b_0'}_x	\rightarrow	\ell^{\infty}_{\gamma(a_0,b_0)},\\
		T	&:L^\frac{4}{3}_tL^{a_0'}_x	\times	L^\frac{4}{3}_tL^{b_1'}_x	\rightarrow	\ell^{\infty}_{\gamma(a_0,b_1)},\\
		T	&:L^\frac{4}{3}_tL^{a_1'}_x	\times	L^\frac{4}{3}_tL^{b_0'}_x	\rightarrow	\ell^{\infty}_{\gamma(a_1,b_0)}.
	\end{aligned}
\end{equation*}	
By Lemma \ref{l:LP} (with $\sigma=1, p_0=p_1=\frac{4}{3}$), $T$ is bounded from
\begin{equation} \label{e:T1bounds}
T:	L^\frac{4}{3}_tL^{a',\frac{4}{3}}_x \times L^\frac{4}{3}_tL^{b',\frac{4}{3}}_x \rightarrow \ell^1_{\gamma(a,b)},
\end{equation}
where
\[
\frac{1}{a} = \frac{1-\theta_0}{a_0} + \frac{\theta_0}{a_1}, \quad \frac{1}{b} = \frac{1-\theta_1}{b_0} + \frac{\theta_1}{b_1}, \quad  0<\theta_0,\theta_1<\theta<1, \quad \theta_0 + \theta_1 = \theta.
\]
In this case, we have used the interpolation identities
\begin{equation*}
(L^\frac{4}{3}_tL_x^{p_0},L^\frac{4}{3}_tL^{p_1}_x)_{\vartheta,\frac{4}{3}} =	L^\frac{4}{3}_tL_x^{p_\vartheta,\frac{4}{3}},
\end{equation*}
where $\frac{1}{p_\vartheta} = \frac{1-\vartheta}{p_0} + \frac{\vartheta}{p_1}$, and $\vartheta \in (0,1)$ (see, for example, \cite{LionsPeetre} or \cite{Cwikel}). Hence,$T$ is bounded as in \eqref{e:T1bounds} for all $(\frac{1}{a},\frac{1}{b})$ in the interior of the triangle with vertices $(\frac{1}{a_0},\frac{1}{b_0}), (\frac{1}{a_0},\frac{1}{b_1}), (\frac{1}{a_1},\frac{1}{b_0}) \in \mathfrak{S}_2$. These vertices were chosen arbitrarily and therefore the boundedness holds in the interior of $\mathfrak{S}_2$. Specialising to $a=b=r \in [4,\infty)$ and using the inclusion $L^{r'} \subset L^{r',\frac{4}{3}}$, we obtain \eqref{e:sumq4}, as desired.
\end{proof}

Following the strategy outlined in Section \ref{subsection:outline}, we have now completed our proof of Theorem \ref{t:coneestimates} in the critical case $\alpha = \alpha^*(q,r)$ whenever $\frac{1}{q} > \frac{d-1}{2}(\frac{1}{2} - \frac{1}{r})$.

%%%%%%%%%%%%%%%%%%
\section{The Schr\"odinger operator} \label{section:schrodinger} 
%%%%%%%%%%%%%%%%%%

Whilst the conical geometry of the singularity region of the multiplier of $\mathcal{C}^\alpha$ played a key role in the proof of Proposition \ref{p:kdeltaest}, captured through the Strichartz estimates for the wave operator $e^{it\sqrt{-\Delta}}$, the other steps were not specific to the cone. Since the theory of Strichartz estimates has been extensively developed, our arguments in Section \ref{section:cone} are readily applicable in other contexts. In this final section, we illustrate this concretely with the Schr\"odinger operator.

Let us define the multiplier operator $\mathcal{S}^\lambda$ by
\[
\mathcal F(\mathcal S^\lambda g)(\xi,\tau)=|\tau-|\xi|^2|^{\lambda} \widehat{g}(\xi,\tau)
\]
and consider estimates of the form
\begin{equation}\label{e:schrodinger} 
\| \mathcal S^\lambda g\|_{L^q_tL^r_x}\lesssim \|g\|_{L^2}.
\end{equation}
By a simple scaling argument one sees that \eqref{e:schrodinger} is true only if $\lambda = \lambda_\circ(q,r)$, where
\begin{equation}
\label{e:schrodinger-scaling}
\lambda_\circ(q,r)=\frac{1}{q}+\frac{d}{2r}-\frac{d+2}{4}. 
\end{equation}
\begin{theorem}\label{t:schrodinger-global} Suppose $d \geq 1$ and $q,r\in [2,\infty)$. Then $\mathcal{S}^\lambda$ is bounded from $L^2_{t,x}$ to $L^q_tL^r_x$ if and only if $\lambda = \lambda_\circ(q,r)$ and $-\frac12<\lambda\le 0$ (equivalently, $\frac{1}{q} > \frac{d}{2}(\frac{1}{2} - \frac{1}{r})$). 
\end{theorem}
Thanks to duality, a straightforward consequence of this is the following Sobolev-type estimate. For all $q,r,\tilde q, \tilde r \in [2,\infty)$ such that $\frac{1}{q} > \frac{d}{2}(\frac{1}{2} - \frac{1}{r})$ and $\frac{1}{\tilde q} > \frac{d}{2}(\frac{1}{2} - \frac{1}{\tilde r})$, the estimate 
\[  
\| f\|_{L^q_tL^r_x} \lesssim  \||\Delta-i \partial_t|^{\gamma} f\|_{L^{\tilde{q}'}_tL^{\tilde{r}'}_x} 
\] 
holds with $\gamma=- \lambda_\circ(q,r)-\lambda_\circ(\tilde q,\tilde r).$ We shall prove Theorem \ref{t:schrodinger-global} as a result of the \emph{critical case} in the forthcoming Theorem \ref{t:schrodinger-local} (which serves as an analogue of Theorem \ref{t:coneestimates}). We also remark that similar Sobolev-type estimates are available in the wave case via Theorem \ref{t:coneestimates}.

Let $\phi$ be a smooth frequency cutoff function whose support is contained in the unit ball in $\mathbb{R}^{d+1}$ centred at the origin  and which satisfies $\phi(0,0)=1$.  Then, we consider the operator $\tilde{\mathcal S}^\lambda$  with localised frequency given by 
\[ 
\mathcal F(\tilde{\mathcal S}^\lambda g)(\xi,\tau)=|\tau-|\xi|^2|^{\lambda} \phi(\xi,\tau) \widehat{g}(\xi,\tau).
\] 
\begin{theorem} \label{t:schrodinger-local}
Suppose $d \geq 1$ and $q,r\in [2,\infty)$.
\begin{enumerate}
\item Suppose $\frac{1}{q} \leq \frac{d}{2}(\frac{1}{2} - \frac{1}{r})$. Then $\tilde {\mathcal S}^\lambda$ is bounded from $L^2_{t,x}$ to $L^q_tL^r_x$ if and only if $\lambda > \lambda^*$.
\item Suppose $\frac{1}{q} > \frac{d}{2}(\frac{1}{2} - \frac{1}{r})$. Then $\tilde {\mathcal S}^\lambda$ is bounded from $L^2_{t,x}$ to $L^q_tL^r_x$ if and only if $\lambda \geq \lambda^*$.
\end{enumerate}
Here,
$
\lambda^* = \lambda^*(q,r) = \max\big\{\frac{1}{q}+\frac{d}{2r}-\frac{d+2}{4}, -\frac{1}{2}\big\}.
$
\end{theorem}
Once Theorem \ref{t:schrodinger-local} is obtained, a scaling argument yields Theorem \ref{t:schrodinger-global}. To establish the critical case $\lambda = \lambda^*$ in the region $\frac{1}{q} > \frac{d}{2}(\frac{1}{2} - \frac{1}{r})$, we can follow the structure of the argument above in Section \ref{section:cone}, and the remaining estimates in non-critical cases can be proved following the argument in \cite{BBGL}; thus, we present only an outline here.
\begin{proof}[Sketch of proof of Theorem \ref{t:schrodinger-local}]  The necessary conditions can be shown in exactly the same way as in \cite[Section 5.2]{BBGL}. For sufficiency, we consider $\mathcal S_\delta$ defined by 
\[ 
\mathcal F(\mathcal S_\delta g)(\xi,\tau) =  \psi\bigg(\frac{|\xi|^2 - \tau}\delta\bigg) \phi(\xi,\tau)\, \widehat g(\xi,\tau),
\]
where $\psi \in C^\infty_c(\mathbb{R})$ is supported in $[\frac{1}{2},2]$. The key estimates are contained in the following proposition.
\begin{proposition}\label{p:sdelta-estimate} Let $0<\delta\ll1$. Then
\begin{equation*}\label{e:sdelta-estimate}
\|\mathcal S_\delta  g\|_{L^q_tL^r_x} \lesssim \delta^{-\lambda^*(q,r)}\|g\|_{L^2}
\end{equation*}
holds whenever $d\geq 1, q,r\in[2,\infty],\frac{1}{q} \geq \frac{d}{2}(\frac{1}{2} - \frac{1}{r})$ and $(q,r,d) \neq (2,\infty,2)$. 
\end{proposition}
One can prove Proposition \ref{p:sdelta-estimate} by following the proof of Proposition \ref{p:kdeltaest}; for $d \geq 3$ this simply means making use of the Strichartz estimates for the Schr\"odinger operator $e^{it\Delta}$ rather than the wave operator (for $d=1,2$, one should follow the additional arguments given in \cite[Section 5]{BBGL}).

Once Proposition \ref{p:sdelta-estimate} is established, the rest of argument leading to Theorem \ref{t:schrodinger-local} is identical to that of the wave case in Section \ref{section:cone} and we omit the details.
\end{proof} 

\begin{acknowledgements}
This work was supported by JSPS Grant-in-Aid for Young Scientists A no. 16H05995 (Bez, Cunanan), JSPS Grant-in-Aid for Challenging Exploratory Research no. 16K13771-01 (Bez), and NRF Republic of Korea no. NRF-2015R1A4A1041675 (Lee). The authors would also like to express their gratitude to Jon Bennett and Susana Guti\'errez for discussions which formed the foundation for this work. 
\end{acknowledgements}

\end{document}